\documentclass[12pt]{article}
\usepackage{amscd,amsthm,amsmath,amsfonts,amssymb,verbatim} 
\usepackage{setspace}

\usepackage{color}
\usepackage{hyperref}
\hypersetup{colorlinks=true,linkcolor=red,}

\newtheorem*{remarks}{Remarks}
\newtheorem{conjecture}{Conjecture}

\theoremstyle{plain}
\newtheorem{thm}{Theorem}
\newtheorem{lem}{Lemma}
\newtheorem{prop}{Proposition}

\def\pp{{\rm PP }}
\def\p{\mathbb{P}}
\def\e{\mathbb{E} \, }

\def\pp{{\rm PP }}
\def\p{\mathbb{P}}
\def\e{\mathbb{E}\,}
\def\sigmabeta{\sigma_{\!\beta}}
\def\mark1{{\ddag}}   
\def\mrk2{{\ddag\ddag}}  
\newcommand{\re}{\textrm{Re}}

\begin{document}

\title{Large deviation asymptotics for a random variable with L\'{e}vy
  measure supported by $[0, 1]$}

\author{Richard Arratia \and  Fred Kochman \and Sandy Zabell}

\date{March 11, 2016}

\maketitle

\begin{abstract}
Asymptotics for Dickman's
number theoretic function $\rho(u)$, as $u \rightarrow \infty$,
were given de Bruijn and Alladi, and later in sharper form by Hildebrand and
Tenenbaum.  The
perspective in these works is that of analytic number theory.  
However, the function
$\rho(\cdot)$ also arises as a constant multiple of a certain probability
density connected with a scale invariant Poisson process, and we observe that
Dickman asymptotics can be interpreted as  
a Gaussian local limit theorem
for the sum of arrivals in a tilted Poisson process, combined with untilting. 

In this paper we exploit and extend  this reasoning to
obtain analogous asymptotic formulas for a class of functions
including, in addition to Dickman's function, the densities of random
variables having L\'{e}vy measure with support contained in $[0,1]$,
subject to mild regularity assumptions.
\end{abstract}

\clearpage
\tableofcontents
\clearpage

\section{Introduction}
\label{1}

Dickman's function $\rho$ is a basic function in analytic number
theory, see \cite[Section III.5]{T}.  It satisfies
\begin{equation}\label{eq1}
 \ \rho(u)=0 \textrm{ for } u<0, \  \rho(u)=1 \textrm{ for } 0\leq u \leq 1,  
\end{equation}
\begin{equation}\label{eq2}
u \rho (u)=\int_{u-1}^u \rho(t) dt \qquad \textrm{for all real } u.
\end{equation}
Write  $\psi(x,y)$ for 
 the number of positive integers $n \leq x$, all of
whose prime factors $p$ satisfy $p \leq y$. In 1930 Dickman
\cite{dickman} showed that for $u>1$,
\[
\rho(u)=\lim_{x \rightarrow \infty} \frac{1}{x}\psi(x, x^{1/u}).
\]
 
Armed with \eqref{eq1} and \eqref{eq2}, a suitable Fourier
representation formula, and the method of steepest descent, Hildebrand
and Tenenbaum \cite{H-T} reprove the classic result of de Bruijn
(1951) and Alladi (1982), see \cite{T},  
that as $u \to \infty$,
\begin{equation}\label{eq3}
  \rho(u)=\sqrt{\frac{\beta^\prime(u)}{2
      \pi}}e^{\gamma-u\beta+C(\beta)}\left\{1+O\left(\frac{1}{u}\right)\right\},
\end{equation}
where $\gamma$ is Euler's constant, $\beta=\beta(u)$ is defined by the formula
\[
e^\beta=1+u \beta,
\] 
and 
\[
C(\beta)=\int_0^\beta \frac{e^t-1}{t} dt=\int_0^1 \frac{e^{\beta t}-1}{t} dt.
\]
(Alladi had improved on the earlier result  of de Bruijn.)

Now consider the scale invariant Poisson process on $(0,\infty)$, with
intensity $(1/x) \, dx$; see \cite{DIMACS}.  Restricting to $(0,1]$,
  we have a Poisson process whose
  arrivals may be labeled in decreasing order, with $1 \ge X_1 > X_2 >
  \cdots >0$, and the sum of these arrivals, $T:= X_1 + X_2 + \cdots$,
  is the random variable characterized by its
moment generating function,
\begin{equation}\label{Dickman M}
  \e e^{\beta T} = \exp\left( \int_0^1 \frac{e^{\beta x} -1}{x}
  \, dx \right).
\end{equation}
Size biasing, see \cite{AGK}, makes it easy to see that 
probability density $f$ of $T$ in place of $\rho$ satisfies \eqref{eq2}.
Obviously $f$ is zero on $(-\infty,0)$, and scale invariance shows
that $f$ is constant on $(0,1)$.  From this it follows that $f$ must
be a constant multiple of $\rho$, and knowing
\[
\int_0^\infty \rho(u) du =e^\gamma,
\]
(see for example \cite[Formulas 5.45 and 5.43]{T})
one sees that
\begin{equation}\label{eq4}
 f (u) =e^{-\gamma} \rho(u).
\end{equation}
Thus we can rewrite \eqref{eq3} as a statement of the asymptotic decay
of the density $f(u)$ as $u \to \infty$ (note Euler's $\gamma$ no
longer appears):
\begin{equation}\label{eq3 f}
f(u)=\sqrt{\frac{\beta^\prime(u)}{2
    \pi}}e^{-u\beta+C(\beta)}\left\{1+O\left(\frac{1}{u}\right)\right\}.
\end{equation}
The function $C$ appearing in these formulas is the cumulant
generating function for $T$, and the similarity to the  large deviation
results of Cram\'{e}r, Chernoff, and their successors (see, for
example,  \cite{Chernoff}) may be
evident.

In this paper, we prove results similar to  
\eqref{eq3 f} for a broader class, those infinitely divisible
distributions whose L\'evy measure is supported on [0,1], subject to
additional mild regularity conditions.  From the perspective of a
probabilist, the novelty of this paper is the derivation of a local
limit theorem, Proposition \ref{new main prop}, for a general case
other than that of classical sums of i.i.d.\ variables, informally
Cram\'{e}r $\beta$-tilts $T_\beta$ of a fixed random variable $T$ as
$\beta \to \infty$.  Via untilting, this leads to asymptotic formulas,
given in Theorems \ref{new theorem 1} and \ref{new theorem 2}, for the
density $f(u)$ as $u \to \infty$, along with a matching asymptotic
formula for the upper tail probability $P(T \ge u)$, for a fixed
random variable $T$ in the class L\'evy [0,1], as defined in Section
\ref{sect levy}.  Our adaptation of the arguments from  
\cite{H-T} and \cite{T} eliminates, at a certain point,
the use of some
Whittaker--Watson species of special function theory which applies
only to the Dickman case,
and substitutes a more robust method.

An important example (covered by our Theorem \ref{new theorem 2}) showing how  variants of \eqref{eq3} arise
naturally
comes from making a ``minor'' change in \eqref{Dickman M}, simply
changing the lower limit of the integral from 0 to  
$a \in
(0,1)$, to get $T \equiv T^{(a)}$ with distribution characterized 
by
\begin{equation}\label{ABT M}
   \e e^{\beta T} = \exp\left( \int_a^1 \frac{e^{\beta x} -1}{x}
  \, dx \right).
\end{equation}
This arises in the study of random permutations, see \cite[Section
  4.3]{ABT book}.   Directly, $f(1)$ governs the asymptotic
probability that a random permutation of $n$ objects has only cycles
of length at least $an$.  Scale invariance leads to $\omega(u)  =
f(1)$ for the case $a=1/u$, with Buchstab's function $\omega$
governing
integers free of small prime factors;  see \cite[Section III.6]{T}.
Scale invariance also leads, for fixed $a \in (0,1)$, to $f(u)$
governing the probability that a random permutation of $n$ objects has
only cycles with lengths in   
$(\frac{a }{u} n,\frac{1}{u}n]$, 
for any $u>1$.

\subsection{L\'evy($\mu$), L\'evy [0,1]}\label{sect levy}

  Let $\mu$ be a nonnegative measure on $(0,\infty)$, such that 
$\int x \mu(dx) \in (0,\infty)$.
We say that the distribution of $T$ is {\bf L\'evy($\mu$)}
  if
\begin{equation}\label{C from mu}
    C(\theta) := \log \e e^{\theta T} =  \int (e^{\theta x}-1) \mu(dx),
\end{equation}
that is, if $T$ has the infinitely divisible distribution
with L\'evy measure $\mu$.  Informally, $T$ is the sum of all arrivals, in the 
Poisson process on $(0,\infty)$ with intensity $\mu(dx)$. 
(There are other infinitely divisible distributions L\'evy($\mu$)
with L\'evy measure supported by $(0,\infty)$, having infinite mean;
by restricting to the finite mean case, we gain both a simplified form for
the L\'evy measure, and the use of moment generating functions, rather
than needing characteristic functions, to specify the distribution.)

We are mainly interested in the case when
the support of $\mu$ is bounded.  Without loss of generality, by
rescaling,  the support of $\mu$ is contained in $[0,1]$, and in
that case, we say that the distribution of $T$ is in the set \mbox{L\'evy [0,1]}.
Obviously, if
the support of $\mu$ is contained in a bounded interval, then
 $\e e^{\theta T}< \infty$
for all $\theta$, and $C$ is defined everywhere.

\subsection{Regularity conditions}

The first regularity condition we impose is that $\mu$ have a density
with respect to Lebesgue measure, say $\mu(dx) = g(x) \, dx$, so that
the distribution of $T$ is determined by
\begin{equation}\label{use g}
   C(\theta) := \log \e e^{\theta T} = \int_0^1 (e^{\theta x}-1) g(x)
   \, dx .
\end{equation}
In this context, the requirement $\int x \, \mu(dx) \in (0,\infty)$
reduces to $\int_0^1 x \, g(x) \, dx \in (0,\infty)$.    Other regularity
conditions are imposed, as needed, for the proofs of Theorems 1 and 2
stated below.  

 There are 
two qualitatively distinct cases, according to whether $\mu((0,1])$ is
finite or infinite.  In the first case,  $ \lambda := \int_0^1
g(x) \, dx < \infty$, so the total
number of arrivals is Poisson distributed with parameter
$\lambda$, 
$\p(T=0)=e^{-\lambda} > 0$, and  the distribution of $T$ has a 
\emph{defective} density $f$, with $\int f(x) dx = 1 -
e^{-\lambda} < 1$.

In the second case, $ \int_0^1 g(x) \, dx = \infty$, and it is not
hard to show that
 the distribution of $T$ has a 
\emph{proper} density $f$, with $\int f(x) \, dx =1$.
An example of the first case is given by \eqref{ABT M}, 
 with  $g(x) = 1/x$ on (a,1] and $a \in (0,1)$, 
and $g(x) = 0 $ on $[0,a]$; and an 
example of this second case, relating to the Dickman function, 
is the density $f = e^{-\gamma} \rho$
in 
\eqref{eq3 f}, with  $g(x) = 1/x$ on (0,1].    

The two cases need different additional
regularity assumptions.     
For the first case, 
with a finite number of Poisson arrivals,  
our main result is given as Theorem \ref{new theorem 1},
which approximates the density $f(u)$ with an
$O(1/u)$ upper bound on the  relative error, as in the
de~Bruijn--Alladi result, and  approximates the upper tail probability
$P(T \ge u)$
with a  $O(1/\sqrt{u})$
upper bound on the relative error. The proof, in Section \ref{5}, of   
Theorem \ref{new theorem 1}
relies on a small amount of Fourier analysis, together with the result
of Proposition~\ref{new main prop}, from Section \ref{3}, which approximates a tilted density 
$f_\beta(t)$ with a uniform bound on the \emph{additive} error.  The proof of 
Proposition~\ref{new main prop} is the most difficult part of this paper,
requiring
estimates for \emph{four} 
different zones of integration.

For the second case, with an infinite number of arrivals,
 arguments requiring no further Fourier analysis will be given, in Section \ref{sect end run},  
letting  us derive Theorem \ref{new theorem 2} which gives
 a  $O(1/\sqrt{u})$
upper bound on the relative error for both the density and the upper
tail probability.

\subsection{Statement of main theorems}

For any integer $k \ge 1$, we say that a function $g(x)$ is 
{\bf piecewise} $\mathbf{C^k}$ on an interval
$[a,1]$, if $[a,1]$ is partitioned into finitely many subintervals on
whose interiors $g$ is $C^k$, and $g$ and all $k$ derivatives all possess finite one-sided
limits at all the endpoints. So in particular, $g$ is bounded on $[a,1]$. We will usually focus
on piecewise $C^2$.

\begin{thm}\label{new theorem 1} 
Assume that the non-negative function $g(x)$ defined on $[0,1]$ is piecewise $C^2$ on $[0,1]$,
and that for some $\epsilon>0$,
we have
\begin{equation}\label{eq101}
g(x)\geq \epsilon \textrm{ on } \left[1-\epsilon,1\right].
\end{equation}
Let the distribution of $T$ be given by \eqref{use g} and let $f$ be
the defective density function for $T$. Given $u>0$, let $\beta
=\beta(u)$ be such that $C'(\beta) =u$; let $\sigma_\beta^2
=C''(\beta)$.  Then as $u \rightarrow \infty$
\begin{equation}\label{eq2032}
f(u)= \frac{e^{C(\beta)-u\beta }}{\sqrt{2 \pi}\,\sigmabeta}
\left(1+O\left(\frac{1}{u}\right)\right)
\end{equation}
and
\begin{equation}\label{tail asymp}
 P(T \ge u) = \frac{1}{\beta} \, \frac{e^{C(\beta)-u\beta }}{\sqrt{2
     \pi}\, \sigmabeta}
 \left(1+O\left(\frac{1}{\sqrt{u}}\right)\right)
\end{equation}
and hence $ P(T \ge u)= (f(u)/\beta) \,
\left(1+O\left(\frac{1}{\sqrt{u}}\right)\right).$
\end{thm}

\begin{thm}\label{new theorem 2}     
Fix $a \in (0,1/4]$.  Let the density function $g: (0,1] \to [0,\infty)$
satisfy:  the restriction of $g$ to $[a,1]$ is   
  piecewise-$C^2$,  $\sup_{0<x \le 1} x g(x) < \infty$, and
$g(x) \ge \epsilon $ on $[1-\epsilon,1]$ for some
  $\epsilon > 0$.
 Given $u>0$, let $\beta
 =\beta(u)$ be such that $C'(\beta) =u$; let $\sigma_\beta^2
 =C''(\beta)$.
Let $f$ be the (possibly defective) density function for $T$.
Then as $u \rightarrow \infty$ 
\begin{equation}\label{eq2032 kludge}
f(u)= \frac{e^{C(\beta)-u\beta }}{\sqrt{2 \pi}\,\sigmabeta}
\left(1+O\left(\frac{1}{\sqrt{u}}\right)\right)
\end{equation}
and
\begin{equation}\label{tail asymp kludge}
 P(T \ge u) = \frac{1}{\beta} \, 
\frac{e^{C(\beta)-u\beta }}{\sqrt{2 \pi}\, \sigmabeta}
\left(1+O\left(\frac{1}{\sqrt{u}}\right)\right)
\end{equation}
and  hence  $ P(T \ge u)= f(u)/\beta \, 
\left(1+O\left(\frac{1}{\sqrt{u}}\right)\right).$
\end{thm}

\begin{remarks}{\normalfont{

1)  The density $g$ can, more generally, be assumed to  have bounded support
    anywhere in the non-negative reals; it is only a useful simplifying
    normalization made here to
    place the support in  $\left[0,1\right]$ with the upper limit of
    supp($g$) equal, in fact, to 1.  So $g$ possesses a discontinuity
    at $x=1$, if nowhere else.

2) {\em Some} condition on the possible growth of $g$ at $0+$ is
necessary for the validity of a strong error term in a local limit
result such as \eqref{eq202} in the underlying
Proposition \ref{new main prop}; 
but the
natural candidate, $\int_0^1 x g(x) \ dx < \infty$, satisfied by $g(x)
= 1/x$ in the Dickman case, does not of itself suffice.  (For example,
fix $0 < \theta < 1$ and let $g(x) dx = \theta \frac{dx}{x}$ on
$(0,1]$.  The density $f$ of the corresponding $T$ satisfies, for all
    $0<x \le 1$, $f(x)=x^{\theta-1} f(1)$, as follows immediately from
    scale invariance, together with the fact that $P($no arrivals in
    $(x,1] ) = \exp(-\int_x^1 g(x)dx)) = x^\theta$.  Thus $f$ is
  unbounded, and the same is true after any tilt and standardization.
  Hence the uniform error estimate in~\eqref{eq202} must fail.)

3)  The boundedness away from 0 at the rightmost boundary of supp$(g)$
    ensures that the tilted measure $e^{\beta x} g(x) dx$ will have
    unbounded mass near 1 as the tilting parameter $\beta$
    increases.  
However, we  believe that our proof could be
    extended to cover the more general case not assuming  
    \eqref{eq101}.  
    
4) We
   conjecture 
 the \emph{conclusions} of 
Proposition \ref{new main prop} 
itself remain
   valid under a far weaker set of hypotheses than any considered here.
In fact, we believe it suffices to assume the intensity measure
   $\nu$ of the Poisson process  $\pp(\nu)$ has a
   non-trivial absolutely continuous part, and the sum $T$ of arrivals
   have finite mean value; see Section~\ref{sect conclusions} for an explicit statement. 
 
}}

\end{remarks}

\section{An easy bulk CLT}

Although the result proved in this section, Theorem \ref{thm easy}, is an easy
exercise, it serves well to introduce our notation for Cram\'{e}r tilts, and
to show how our local limit results are much more delicate than
the bulk central limit theorem; see also Conjecture
   \ref{conjecture 1} in Section~\ref{sect conclusions}.  We give an
example, after  Theorem \ref{thm easy}, to show that the 
conclusion of the theorem can fail without the hypothesis that the support of $\mu$ is bounded.   

For $\beta \in (-\infty,\infty)$, let $T_\beta$ be distributed as the Cram\'{e}r 
$\beta$-tilt of $T$, that is, $\p(T_\beta \in dx) = \p(T \in dx) \, e^{\beta
  x} / \e e^{\beta T}$, so that with $C_\beta$ defined by 
$C_\beta(\cdot):=C(\cdot +\beta)-C(\beta)$, the distribution of $T_\beta$ has
cumulant generating function $C_\beta$.  Thus, for our $T$ as given by
\eqref{C from mu},
\begin{equation}\label{cramer} 
\e e^{\theta T_\beta} = \exp(C_\beta(\theta))  
  =  \exp\left( \int (e^{\theta x}-1) \, e^{\beta x} \mu(dx)\right).
\end{equation}
Informally, $T_\beta$  is the sum of all arrivals, in the 
Poisson process on $(0,1]$ with intensity $e^{\beta x} \mu(dx)$.

\begin{thm}\label{thm easy}
Let $T$ be distributed as per \eqref{C from mu}, with $\mu$ supported
by $(0,1]$ and $\int x \mu(dx) \in (0,\infty)$.
Consider the Cram\'{e}r tilts of $T$, as given by \eqref{cramer}.
As $\beta$ grows so that $\e T_\beta \to \infty$,
$(T_\beta - \e T_\beta)/\sqrt{ {\rm Var }\,  T_\beta}$
converges in distribution to the standard normal.
\end{thm}

\begin{proof}
 
 The mean and variance of $T_\beta$, call them $u(\beta)$ and
$\sigma^2(\beta)$, are given by $u(\beta)= C_\beta'(0) = \int x
e^{\beta x} \mu(dx)$ and $\sigma^2(\beta) = C_\beta''(0) = \int x^2
e^{\beta x} \mu(dx)$.  Clearly, $\infty > u(\beta) \ge \sigma^2(\beta)
\to \infty $ as $\beta \to \infty$, using the hypotheses that $\int x
\mu(dx) \in (0,\infty)$ and that the support of $\mu$ is contained in
$[0,1]$.  It is easy to see, from the explicit integrals for
$u(\beta)$ and $\sigma^2(\beta)$, that with $x_0 :=
\sup($support($\mu)) \in (0,1]$, $\sigma^2(\beta)/u(\beta) \to
x_0$ as $\beta \to \infty$.  To set up use of the Lindeberg-Feller
central limit theorem, 
fix a sequence $\beta(n) \to \infty$, and
take a triangular array where the $n$th row has
$m(n)$ i.i.d. mean zero entries, for $m(n) = \lceil u(\beta(n)) \rceil$,
and the sum of these $m(n)$ entries is $(T_\beta(n) - \e T_{\beta(n)})$.
In other words, each entry in row $n$ is distributed as $Y_n - \e Y_n$,
where
$$
   \e e^{Y_n} =  \exp(\left(\int (e^{\theta x}-1) \, \frac{1}{m(n)} e^{\beta
   x} \mu(dx) \right).
$$
Note that $m(n) \sim u(\beta(n))$ as $n \to \infty$.
The sum of the variances for the $n$th 
row is $ \sigma^2(\beta(n))$, 
with  $ \sigma^2(\beta(n)) \sim x_0 \, u(\beta) \sim x_0 \, m(n)$.
The hypothesis of the Lindeberg-Feller theorem, for any triangular array having $m(n)$ independent entries each distributed as $Y_n$, with total variance $\sigma^2(\beta(n))$ for the $n$th row, is that
for fixed $\varepsilon>0$,
\begin{equation}\label{wanted}
   m(n)  \, \e( (Y_n - \e Y_n)^2; | Y_n - \e Y_n| > \varepsilon \,
     \sigma(\beta(n))) = o\left(\sigma^2(\beta(n))\right);
\end{equation}
hence for our setup we need only show that
$$
\e( (Y_n - \e Y_n)^2; | Y_n - \e Y_n| > \varepsilon \,
     \sigma(\beta(n))) = o\left( 1 \right)
$$
as $n \to \infty$.     

For sufficiently large $n$, $\varepsilon \, \sigma(\beta(n)) >1$, and
since $Y_n \ge 0$ and $\e Y_n \le 1$, for these sufficiently large $n$ we have
$$ \e( (Y_n - \e Y_n)^2; | Y_n - \e Y_n| > \varepsilon \,
\sigma(\beta(n))) = \e( (Y_n - \e Y_n)^2; Y_n - \e Y_n > \varepsilon
\, \sigma(\beta(n))),
$$
which in turn is at most 
$$  \e( Y_n^2; Y_n > \varepsilon \, \sigma(\beta(n))).
$$

Finally, \cite[Theorem 1.2 and Section 6]{ArratiaBaxendale} assert
that \emph{any} random variable $X$ in
L\'evy[0,1], with $\e X \le 1$, satisfies, for all $t \ge 1,$ $ \p(X
\ge t) \le 1/\Gamma(1+t)$.  Our $Y_n$ is of this form, so the
upper bound on the upper tail probability
gives $\e (Y_n^2; Y_n \ge x) =  x^2 \, \p(Y_n \ge x) +\int_x^\infty 2t \, \p(Y_n \ge t) \ dt
\le x^2 / \Gamma(1+x) + \int_{t \ge x} 2t/\Gamma(1+t) \ dt
= o(1)$, using   $x = \varepsilon \, \sigma(\beta(n)) \to \infty$.
 
\end{proof}

\noindent {\bf Example}  Take $\mu$ to be the measure on $(0,\infty)$
with $\mu(dx) = e^{-x}/x \ dx$;  
this is known as the \emph{Moran subordinator},  see for example
\cite[Section 9.4]{Kingman}.  With $T$ and $T_\beta$ as given by
by \eqref{C from mu} and \eqref{cramer}, $T$ has the standard
exponential distribution, and for $\beta<1$, $T_\beta$ has the
exponential distribution with mean $1/(1-\beta)$.  For any $u \in
(0,\infty)$, one can solve $\e T_\beta=u$, and as $u \to \infty$, we
have $\beta(u) \to 1$ but there is no rescaling and centering of $T_\beta$
which converges to the normal distribution.

\section{Preliminaries}\label{s2}

In this section we fix notation and prove some preliminary 
lemmas about the distribution of the sum of arrivals, $T$, in a
Poisson process $\pp(\nu)$ on the non-negative reals, with intensity
measure $\nu$ satisfying
\[
d\nu=g(x)dx
\]
for some density function $g(x)$ satisfying certain subsets of the hypotheses of Theorem 
\ref{new theorem 1}. 
 For  omitted
proofs or definitions pertaining to Poisson processes, we refer the
reader to \cite{Kingman}.
 
For any intensity measure $\nu$ with support in $\left[0,1\right]$,
define
\begin{equation}\label{eq101b}
C(z):=\int_0^1(e^{zx}-1)d\nu (x),
\end{equation}
and let $T$ denote the sum of arrivals in $\pp(\nu)$.  When
$d\nu=e^{\beta x}g(x)\,dx$ for some fixed $g(x)$, we may write $T_\beta$
and $C_\beta$, but in Lemma~\ref{p1} 
below, we
suppress dependence on $\beta$, to avoid clutter.  We  sometimes
specifically single out the case $\beta=0$, i.e., the untilted
measure, with the subscript ``$0$''.  Thus
\begin{equation}\label{eq102}
C_0(z)=\int_0^1\left(e^{zx}-1\right) g(x) dx
\end{equation}
and
\[
T_0 \textrm{ is the sum of arrivals in } \pp(g(x)dx).
\]  
Trivially, for $d \nu=e^{\beta x}g(x) dx$, we have 
\begin{equation}\label{eq103}
C(z)=C_0(z+\beta)-C_0(z).
\end{equation}
\begin{lem}\label{p1}
Let $\nu$ and $T$ be as just discussed.  Then  

a) $Ee^{zT}=e^{C(z)}$

b) $ET=\int_0^1 x d\nu(x)$, and
 
c) var $(T)=\int_0^1 x^2 d \nu (x)$.
\end{lem}
\begin{proof}
These are parts of Campbell's theorem, valid for very general
intensity measures $\nu$;  see \cite{Kingman}.
\end{proof}
\begin{lem}\label{p2}
Let $d \nu=e^{\beta x} g(x) dx $, where $g(x)$  
is nonnegative and bounded and, for some $\varepsilon > 0$, satisfies \eqref{eq101}.  Then

a)	$Ee^{zT_\beta}=e^{{C_0}(z+\beta)-C_0(\beta)}$.

b)	There are constants $0<K_1<K_2$ (depending on $g$), such that
$C_0(\beta)$, together with any finite collection of integrals
$\int_0^1 x^k e^{\beta x} g(x) dx$ for $k=0, 1, 2,  \ldots $ all lie
between $K_1 \frac{e^\beta}{\beta}$ and $K_2\frac{e^{\beta}}{\beta}$
for $\beta$ sufficiently large. (The constants also depend on the
particular finite collection.)
If only $xg(x)$ is bounded, instead of $g(x)$, this still holds for $k = 1,2,\dots$.

c)  There are (different) constants $0<K_1<K_2$ such that for $\beta$
sufficiently large, the ratio of any pair of integrals from part b)
lies between $K_1$ and $K_2$. 

\end{lem}
\begin{proof}
a) is immediate from \eqref{eq103} and c) is immediate from b).  As
for b), when $g(x)$ is bounded our hypotheses imply that for some $K, \epsilon>0$, for any
non-negative function $h(x)$ we have
\[
\epsilon \int_{1-\epsilon}^{1} h(x) e^{\beta x}dx \leq \int_0^1 h(x)
e^{\beta x} g(x) dx \leq K\int_0^1 h(x) e^{\beta x} dx.
\]
The rest is integration by parts.

If only $xg(x)$ is bounded then to get the rightmost inequality, for any non-negative $h(x)$
for which $h(x)/x$ is bounded on $[0,1]$ (such as $h(x) = x$) rewrite the
middle integrand as $h(x)e^{\beta x} g(x) = (h(x)/x)e^{\beta x} (xg(x))$ and proceed from there.
\end{proof}

\begin{lem}\label{p3}
Let $d \nu =e^{\beta x} g(x) dx$, where $g$ is nonnegative and bounded and, 
for some $\varepsilon > 0$, satisfies \eqref{eq101}.  Then
 
a) $C_0 (\beta)$, along with all the integrals $\int_0^1 x^k e^{\beta
  x} g(x) dx$ for $k= 0, 1, 2, \ldots$, grows to $\infty$ as $u(\beta)
\rightarrow \infty$.  

b) The following statements only require $xg(x)$ bounded, not $g(x)$:\\ 
For $\beta >0$, the function $u(\beta):=ET_\beta$ satisfies
$\frac{du}{d \beta}=\textrm{ var } (T_\beta)>0$ and, so, is monotone
increasing and hence invertible.  The inverse function $\beta (u)$
satisfies $\frac{e^\beta}{u}\rightarrow \infty$ as $u \rightarrow
\infty$ (or $\beta \rightarrow \infty$), but for any $\epsilon>0$,
$\frac{e^\beta}{u^{1+\epsilon}}\rightarrow 0$.

c) $P(T_\beta=0)=e^{-\int_0^1 e^{\beta x} g(x) dx}
=O\left(e^{-Ku}\right)$ for some $K>0$, where $u=ET_\beta$.
\end{lem}
\begin{proof}
This is all just a corollary of Lemma~\ref{p2}.  a) follows at
once from Lemma~\ref{p2}(b).  As for b), the assertions about
$u(\beta)$ are immediate.

To confirm the growth properties of $e^{\beta}$, if we had
$e^{\beta}/u <K$ for arbitrarily large $u$, for some $K$, then also we
would have
\[
u=ET_\beta<K_2 \frac{e^{\beta}}{\beta} \Rightarrow u
<\frac{KK_2u}{\beta}\Rightarrow 1<\frac{KK_2}{\beta}
\]
for arbitrarily large $\beta$, which is impossible.

Also, suppose that for some $\epsilon>0$ we have
$\frac{e^\beta}{u^{1+\epsilon}}>K>0$, for some $K$, for arbitrarily
large $u$.  Pick $0<\zeta <1$ such that $\left(1+\epsilon\right)
\left(1-\zeta\right)>1$.  Since asymptotically
$\frac{e^\beta}{\beta}>e^{(1-\zeta)\beta}$ and also $u
>K_1\frac{e^\beta}{\beta}$, we would have $u>K_1
e^{(1-\zeta)\beta}>u^{(1-\zeta)(1+\epsilon)}K_1K^{1-\zeta}$ for
arbitrarily large $u$, which is also impossible.

 Finally, the formula for $P(T_\beta=0)$ is standard Poisson theory,
 and the bound follows from the definition of $u$ with
 Lemma~\ref{p1}(b) and~\ref{p2}(b).
\end{proof}
 
To conclude, we show that when the density $g$ is piecewise $C_k$
for $k \ge 1$, 
the random variable $T_\beta$ does, in fact have a density
function $f_\beta$, possessing a certain degree of regularity.  
Strictly speaking we should
refer to $f_\beta$ as a \emph{defective} density, since the
distribution of $T_\beta$ has a positive atom at 0.  Namely, under our
hypotheses, for $\beta \geq 0$ we have
\[
\mu:=\nu \left\{ \left[0,1\right]\right\} =\int_0^1 e^{\beta x} g(x)
dx < \infty.
\]
So in our Poisson process the probability of no arrivals at all,
$P(T_\beta=0)$, is $e^{-\mu}>0$.

However, conditional on there being any arrivals at all, $T_\beta$
does possess a conditional density.  In fact, we have the
following:
\begin{lem}\label{p4}
Let the density function $g(x)$ be piecewise $C^k$ on $[0,1]$, for $k \ge 1$.

a) Conditional on $T_\beta>0$, $T_\beta$ possesses a (conditional)
density $h$.

b) $h$ is piecewise continuous, and for $t>1$, $h(t)$ is continuous.

c) For $t>2$, $h(t)$ is $C^1$.
\end{lem}
\begin{proof}
Let $q(x)$ be the probability density on $[0,1]$ proportional 
to $e^{\beta x}g(x)$.

a)	Conditional on exactly $k$ arrivals, for $k>0$, the density 
of $T_\beta$ is the $k$-fold convolution  $q* \cdots *q$.  
(See \cite{Kingman} for this standard result.)  Note that these 
convolution products are uniformly bounded by $\sup(q)$, 
since for any pair $q_1$ and $q_2$ of probability densities on
$\mathbb{R}$ we have
\begin{align*}
\sup(q_1 * q_2) &=\sup \left\{ \int_{-\infty}^{\infty} q_1 (t-s) q_2
(s) ds \right\}\\ &\leq \sup (q_1) \cdot \int_{-\infty}^{\infty} q_2
(s) ds =\sup (q_1).
\end{align*}

Let $p_k$ be the conditional probability of $k$ arrivals, given that
there is at least one.  Because the $q^{*k}$ are uniformly bounded,
the sum $ \sum_{k=1}^\infty p_k q^{*k} (x) $
exists and is measurable;   and since the conditional probability that
$a<T_\beta \leq b$, given at least one arrival, is
\[
\sum_{k\geq 1} p_k \int_a^b q^{*k}(x)dx,
\]
it follows that $h(x)$ is a conditional density for $T_\beta$.

b) It is an exercise to check inductively that for $k \geq 2$, the
$k$-fold convolution of $k$ copies of $q$ is continuous.  So by
uniform convergence, the function 
$$
h(x) = \sum_{k \geq 2}^\infty p_k q^{*k}
$$
is continuous.  Since supp$(q) \subset \left[ 0,1 \right] $ it follows
that $h=q+\sum_{k \geq 2}^{\infty} p_k q^{*k}$ is continuous for
$t>1$, and piecewise continuous for $t \leq 1$.

c) It is another exercise\footnote{see, e.g. exercise 2.37 on page 128 of \cite{Kammler}} 
to show that since $q$ is piecewise $C^1$ and $q^{*k}$
is continuous for $k\geq 2$, $q^{*(k+1)}=q*(q^{*k})$ is $C^1$.  We claim, further, that
\[
h_3=\sum_{k=3}^{\infty} p_k q^{*k}
\]
is $C^1$.  If so, then since supp$(p_1q+p_2q*q)\subset
\left[0,2\right]$, it follows that $h$ is $C^1$ for $t>2$.

We show that $h_3$ is $C^1$.  Since
$\frac{d}{dt}\left(q*q^{*k}\right)=q^\prime * q^{*k}$, it follows by
the argument in part a) that since $q^\prime$ is bounded, the
derivatives of $q^{*(k+1)}$ are uniformly bounded in $k$.  Let $d(t)$
be the series formed from term by term derivatives.  Then
\begin{align*}
&\left|\left(h_3(t+\delta)-h_3(t)\right)/\delta-d(t)\right|\\ &\leq
  \left| \sum_{k=3}^{k_0} p_{\beta} \left( \left(
  q^{*k}\right)^{\prime}(t_k)-\left(
  q^{*k}\right)^{\prime}(t)\right)\right| +\left| \sum_{k\geq
    k_0}\left(\left(q^{*k}\right)^{\prime}(t_k)-(q^{*k})^{\prime}(t)\right)\right|
\end{align*}
 where $t \leq t_k \leq t+\delta$, by the mean value theorem.  Because
 of all the boundedness and the convergence of $\sum p_k$, the tail
 can be made arbitrarily small, independent of $\delta$, with
 sufficiently large $K_0$, and then the initial segment can also be
 made arbitrarily small with sufficiently small $\delta$.  So $h_3$ is
 $C^1$ with $h_3^\prime (t)=d(t)$.
\end{proof}
Writing $p_0=P(T_\beta=0)$, the probability measure underlying
the distribution of $T_\beta$ can be written as 
\[
p_0 \delta_0(t)+(1-p_0) h(t) \ dt.
\]
The term $(1-p_0)h$, with total mass $1-p_0<1$, is referred to as a
defective density.  
\newline

\section{Statement and proof of Proposition \ref{new main prop}}\label{3}

Let $g(x)$ be a bounded density function  satisfying the 
hypotheses of Theorem \ref{new theorem 1}.  Let $T_\beta$ be as given by \eqref{cramer}. 
Let
\begin{equation}\label{eq201}
Y  \equiv Y_\beta :=\frac{T_\beta-E_\beta}{\sigmabeta}
\end{equation}
be the standardized version of $T_\beta$.  Since $T_\beta$ possesses
an atom at 0, $Y$ possesses an atom at $\frac{-E_\beta}{\sigmabeta}$.
By Lemma \ref{p4} we know that $Y$ possesses a defective density function $f_Y(y)$.  We will prove
the following result:

\begin{prop}\label{new main prop}   
If $\beta \rightarrow \infty$, then  $E_\beta \to \infty$ and
\begin{equation}\label{eq202}
f_Y(y)=\frac{1}{\sqrt{2 \pi}}e^{-y^2/2} +O\left( \frac{1}{\sqrt{E_\beta}}\right)
\end{equation}
uniformly in $y$.  Further, for $y=0$ we have the stronger statement 
\begin{equation}\label{eq203}
f_Y(0)=\frac{1}{\sqrt{2\pi}}+O\left(\frac{1}{E_\beta}\right).
\end{equation}
\end{prop}

\begin{proof}
This proof is inspired by the proof of Theorem 2.1 in \cite{H-T}, though we 
we cannot use formulas involving $\hat{\rho}(s)$, and associated quantities, 
that play a significant role in their treatment.

Write
\begin{equation}\label{eq205}
u=ET_\beta, \ \ \sigmabeta^2 = {\rm var}\, T_\beta,
\ \ \alpha=1/\sigmabeta.
\end{equation}
Since $T_\beta$ has an atom at 0, $Y=\alpha\left(T_\beta-u\right)$
inherits an atom at $-\alpha u$, with probability mass
\[
P\left(T_\beta=0\right)=e^{-\int_0^1 e^{\beta x} g(x) dx}.
\]
Therefore we can calculate the characteristic function of $Y$ as
\[
\varphi(t):=Ee^{itY}=P(T_\beta=0)e^{-it\alpha u}
+\int_{-\infty}^\infty e^{ity}f_Y(y)dy.
\]
But since by Lemma~\ref{p2}(a) we also have
\begin{align*}
\varphi(t)&=Ee^{it(\alpha T_\beta-\alpha u)}\\ &=e^{C_0(\beta +i
  \alpha t)-C_0(\beta)-i \alpha t u}
\end{align*}
we find that the Fourier transform of $f_Y$ takes the form
\[
\hat{f_Y}(t) =\int_{-\infty}^{\infty}e^{-i t y}f_Y(y)dy=e^{C_0(\beta-i
  \alpha t)-C_0(\beta)+i \alpha t u}-P(T_\beta=0) e^{i\alpha t u}.
\]
Since $f_Y$ possesses discontinuities in the interval $\left[-\alpha
  u, 1-\alpha u\right]$, $\hat{f_Y}$ cannot be an $L^1$ function.
Nevertheless, we have
\[
f_Y(y)=\lim_{A \rightarrow \infty} \frac{1}{2\pi} \int_{-A}^{A}
e^{iyt} \hat{f_Y}(t)dt
\]
\begin{equation}\label{eq206}
=\lim_{A \rightarrow \infty} \frac{1}{2\pi} \int_{-A}^{A}
e^{iyt}\left\{e^{C_0(\beta-i \alpha t)-C_0(\beta) +i \alpha t
  u}-P(T_\beta=0)e^{i \alpha t u}\right\} dt,
\end{equation}
valid wherever $f_Y$ is differentiable.  But from Lemma~\ref{p4}
and Formulas~\eqref{eq2035} and \eqref{eq2036}, we know that for any
$y$, $f_Y^{\prime}(y)$ exists for $\beta$ sufficiently large, and this
is all we will need.  (See \cite[Chapter 6, Section 5]{Walker},
particularly Theorem~5.13 on page 243, for the applicable inversion
formula.)

We now proceed to evaluate the right hand side of \eqref{eq206}, for finite
$\beta$.  We partition the domain of integration, in \eqref{eq206},
into four concentric zones around $t=0$ and work on them separately.
Choose and fix $r > 3\pi^2/4$.  The zones are
\begin{align*}
& \textrm{Zone 0: $-R(u)\leq t\leq R(u)$, where $R(u)=\sqrt{r\log u}$}.
\\ & \textrm{Zone 1:
    $\left[-\pi \sigmabeta, -R(u)\right] \cup \left[R(u), \pi
      \sigmabeta \right]$}\\ & \textrm{Zone 2: $\left[-\beta
      \sigmabeta^2, -\pi \sigmabeta\right] \cup \left[\pi \sigmabeta,
      \beta \sigmabeta^2 \right]$}\\ & \textrm{Zone 3: $\left[-A,
      -\beta \sigmabeta^2\right] \cup \left[ \beta \sigmabeta^2,
      A\right]$ (where $A \rightarrow \infty$).}
\end{align*}

\begin{prop}\label{p5}
The contributions to \eqref{eq206} from the above four zones, are, respectively,
\begin{align*}
\textrm{Zone } 0&: \frac{1}{\sqrt{2
    \pi}}e^{-y^2/2}+O\left(\frac{1}{\sqrt{u}}\right) \qquad \textrm{if
  $y \ne 0$};\\ 
 &\ \ \frac{1}{\sqrt{2
    \pi}}+O\left(\frac{1}{u}\right) \qquad \textrm{if
  $y=0$}.\\
 \textrm{Zone } 1 &: o\left(\frac{1}{u}\right). 
\\ 
\textrm{Zone } 2&:
o\left(\frac{1}{u^k}\right) \qquad \textrm{for any
  $k>0$.}\\
 \textrm{Zone } 3&: o\left(\frac{1}{u^k}\right) \qquad
\textrm{for any $k>0$.}
\end{align*}
\end{prop}

Because the integral is the sum of the four contributions,
this 
will prove Proposition \ref{new main prop}.
 \vspace{1pc}

\noindent \emph{Proof of Proposition~\ref{p5}.\ }
For the first three zones we can neglect the term
\[
P\left(T_\beta=0\right) e^{i \alpha t u} 
\]
in the integrand of \eqref{eq206}.  In fact, the following is true:  

\begin{lem}
We have
\[
P\left(T_\beta=0\right) \int_{-\beta \sigmabeta^2}^{\beta \sigmabeta^2} e^{i(y+\alpha u)t} dt=O\left(e^{-K u}\right) \textrm{for some $K>0$}.
\]
\end{lem}

\begin{proof}
The integral is absolutely bounded by $2 \beta \sigmabeta^2$, and by
Lemmas~\ref{p2} and~\ref{p3}
\[
P(T_\beta=0)\beta \sigmabeta^2 \leq K^\prime e^{-K^\prime u}u^2 \log u
\]
for some $K^\prime>0$; so let $K=K^\prime/2$.\end{proof}

\noindent \textit{Remark}:  We do  not actually need
the boundedness assumption on
$g(x)$ in the analysis of the first three zones---just the  finiteness of
$ET_0$ will do. We most definitely use the boundedness in Zone $3$,
however, and at present see no way to dispense with that requirement, or
something close to it.
\\

\noindent \textbf{Zone 0:} We will show that modulo the neglected
$O\left(e^{-k u}\right)$ term we have
\begin{equation}\label{eq208}
\frac{1}{2 \pi} \int_{-R(u)}^{R(u)} e^{iyt} \hat{f}_Y(t)dt=\frac{1}{2
  \pi}\int_{-R(u)}^{R(u)} e^{iyt}\left[e^{C_0(\beta -i \alpha
    t)-C_0(\beta)+i \alpha t u}\right] dt.
\end{equation}
Since $u=\int_0^1 xe^{\beta x} g(x) dx$, by Taylor's theorem we may
write
\begin{align} 
C_0\left(\beta-i \alpha t\right)&-C_0(\beta) + i\alpha t u=\int_0^1
\left\{e^{(\beta-i\alpha t)x}-e^{\beta x}+i \alpha t x e^{\beta
  x}\right\} g(x) dx \label{eq209} \\ &=\int_0^1 e^{\beta x}
\left\{\frac{-\alpha^2 t^2 x^2}{2}+\frac{i^3 \alpha3 t^3
  x^3}{3!}+\alpha^4 t^4 x^4\cdot O(1)\right\} g(x) dx \label{e31}
\end{align}
as $\alpha t x \rightarrow 0$. Since $\alpha R(u)\leq
K\sqrt{\frac{\log u}{u}}$ for some $K>0$, and $|x|\leq 1$, we do have
$\alpha t x \rightarrow 0$ as $\beta \rightarrow \infty$, uniformly in
$t$ within Zone 0.  So by Lemma~\ref{p2}, \eqref{eq209} becomes
\begin{equation}\label{eq210}
-\frac{t^2}{2}+\alpha t^3 O_3(1) +\alpha^2 t^4 O_4(1)
\end{equation}
where $O_3(1)$ is independent of $t$, not just uniform.  Then Taylor's
theorem applied again tells us
\begin{equation}\label{eq211}
e^{C_0(\beta- i \alpha t)-C_0(\beta)+i \alpha t u}=e^{-t^2/2}
\left(1+\alpha t^3 O_3(1)+\alpha^2
t^4 O_4(1)\right)+O\left(\left[\alpha t^3 +\alpha^2
  t^4\right]^2\right).
\end{equation}
Since $\alpha R^3(u)\leq K\sqrt{\frac{\log^3 u}{u}} \rightarrow 0$ as
$u \rightarrow \infty$ and $\alpha^2 R^4(u) \leq K \frac{\log^2 u}{u}
\rightarrow 0$ as $u \rightarrow \infty$, the remainder terms will
die.

Inserting the terms of \eqref{eq211} into \eqref{eq208}, one by one,
we find, first
\begin{equation}\label{eq2115}
\frac{1}{2 \pi}\int_{-R(u)}^{R(u)} e^{i y t} e^{-t^2/2}
dt=\frac{1}{\sqrt{2\pi}} e^{-y^2/2}+O\left(\frac{1}{R(u)}
e^{-R^2(u)/2}\right)
\end{equation}
\[
=\frac{1}{\sqrt{2 \pi}}e^{-y^2/2} +O\left( \frac{1}{\sqrt{\log u}}
\frac{1}{u^{r/2}}\right);
\]
since $r>2$, the remainder is subdominant.  (See the remarks in
the analysis of Zone 1 concerning the size of $r$.)

Next, if $y=0$, then $\frac{1}{\sqrt{2 \pi}}\int_{-R(u)}^{R(u)}
e^{iyt} e^{-t^2/2} \alpha t^3 dt \cdot O_3(1)=0$ by symmetry.
Otherwise, since all moments of $e^{-t^2/2}dt$ are finite, this term
contributes $\alpha O(1)=O\left(\frac{1}{\sqrt{u}}\right)$.
Similarly, the $t^4$ term, as well as all subsequent terms, can
contribute at most $O(\alpha^2)=O\left(\frac{1}{u}\right)$.

This completes the analysis in Zone 0.  For Zones 1 and 2, we need a
lemma adapted from the analysis of inequalities (2.20) in \cite{H-T}, in which our $C_0(\beta)$
plays the role of their $E(\xi)$.  

\begin{lem}\label{L2}
For $\tau$ real, let
\[
H(\tau)=C_0(\beta) -\re \left\{C_0(\beta-i \tau)\right\}.
\]
Then
\begin{enumerate}
\item[a)] For $|\tau| \leq \pi, H(\tau) \geq \frac{2
  \tau^2}{\pi^2}\sigmabeta^2 $
\item[b)] For $|\tau| \geq \pi $ and $\beta$ sufficiently large,
\[
H(\tau)>\frac{\epsilon \pi^2}{8} \frac{e^\beta}{\beta^3},
\]
where $\epsilon$ is the lower bound for $g(x)$ near $x=1$.
\end{enumerate}
\end{lem}

\begin{proof}

\begin{enumerate}
\item[a)] \begin{align*} H(\tau)&=\int_0^1 \left(e^{\beta
    x}-1\right)g(x)dx -\int_0^1\left(e^\beta \cos \tau x-1\right) g(x)
  dx\\ &=\int_0^1 e^{\beta x} \left( 1- \cos \tau x\right) g(x)
  dx\\ &\geq \int_0^1 e^{\beta x} \frac{2 \tau^2 x^2}{\pi^2} g(x)
  dx\\ &=\frac{2 \tau^2 }{\pi^2} \sigmabeta^2 \ .
\end{align*}
\item[b)]
\begin{align*}
H(\tau)&=\int_0^1 e^{\beta x} \left(1-\cos \tau x\right) g(x)
dx\\ &\geq \epsilon \int_{1-\epsilon}^1 e^{\beta x} \left(1-\cos \tau
x\right) dx\\ &=\re \left\{ \epsilon \int_{1-\epsilon}^1 e^{\beta x}
\left( 1-e^{i\tau x}\right)dx
\right\}\\ &=\frac{\epsilon}{\beta}\left(e^\beta-e^{(1-\epsilon)\beta}\right)-\re\left\{
\left[ \frac{\epsilon e^{i\theta}}{|\beta + i \tau|}\right] \left[
  e^{\beta + i\tau
  }-e^{(1-\epsilon)(\beta+i\tau)}\right]\right\}\\ &\qquad
\left(\textrm{where $\theta= \arg \frac{1}{\beta +
    i\tau}$}\right)\\ &\geq
\frac{\epsilon}{\beta}\left(e^\beta-e^{(1-\epsilon)
  \beta}\right)-\epsilon
\frac{e^\beta+e^{(1-\epsilon)\beta}}{|\beta+i\tau|}\\ &=\frac{\epsilon
  e^\beta}{\beta}\left[1-e^{-\epsilon \beta}-\frac{1+e^{-\epsilon
      \beta}}{\sqrt{1+\pi^2/\beta^2} } \right] \\ &\geq \frac{\epsilon
  e^\beta}{\beta}\left[1-e^{\epsilon \beta}-\left(1+e^{-\epsilon
    \beta}\right)
  \left(1-\frac{\pi^2}{4\beta^2}\right)\right]\\ &=\frac{\epsilon
  e^\beta}{\beta}\left[\frac{\pi^2}{4 \beta^2}-\left(2-\frac{\pi^2}{4
    \beta^2}\right) e^{-\epsilon \beta}\right]\\ &>\frac{\epsilon
  e^\beta}{\beta}\cdot \frac{\pi^2}{8 \beta^2}
\end{align*}
for $\beta$ sufficiently large.  

\end{enumerate} \end{proof}

Now we do the Zone 1 and Zone 2
estimates.  Both involve straightfoward estimates.
\\

\noindent \textbf{Zone 1:} For the upper half of Zone 1 we write
\begin{align*}
&\left| \int_{R(u)}^{\pi \sigmabeta}
  e^{iyt}\left\{e^{C_0(\beta-i\alpha \tau)-C_0(\beta)+i\alpha \tau
    u}\right\}dt\right|\\ &\leq \int_{R(u)}^{\pi \sigmabeta}
  e^{-H(\alpha t)}dt\\ &=\sigmabeta \int_{R(u)/\sigmabeta}^\pi
  e^{-H(\tau)}d\tau\\ &\leq \sigmabeta \int_{R(u)/\sigmabeta}^{\pi}
  e^{-\frac{2 \sigmabeta^2}{\pi^2}\tau^2}d \tau\\ &\leq \pi \sigmabeta
  e^{-2R(u)^2/\pi^2}\\ &=\frac{\pi
    \sigmabeta}{u^{2r/\pi^2}}=O\left(\frac{1}{u^{2r/\pi^2-1/2}}\right)=o\left(\frac{1}{u}\right)\
\end{align*}
since $r>\frac{3\pi^2}{4}$ and $\sigmabeta =O\left( \sqrt{u} \right)$.
(If $r> (2k+1)\pi^2/4$, we can even get $o\left(\frac{1}{u^k}\right)$
for any desired $k$.)  The estimate for the lower half is the same.
\\

\noindent \textbf{Zone 2:} For the upper half we write
\begin{align*}
&\left| \int_{\pi \sigmabeta}^{\beta \sigmabeta}
  e^{iyt}\left\{e^{C_0(\beta-i\alpha \tau)-C_0(\beta)+i\alpha \tau
    u}\right\}dt\right|\\ &\leq \int_{\pi \sigmabeta}^{\beta
    \sigmabeta} e^{-H(\alpha t)}dt\\ &=\sigmabeta \int_{\pi}^{\beta
    \sigmabeta} e^{-H(\tau)} d\tau\\ &\leq \sigmabeta \int_\pi^{\beta
    \sigmabeta} e^{-\frac{\epsilon
      \pi^2}{8}\frac{e^\beta}{\beta^3}}d\tau\\ &\leq \beta \sigmabeta
  e^{-\frac{\epsilon \pi^2}{8}\frac{e^\beta}{\beta^3}} \qquad
  \textrm{for $\beta$ sufficiently large},\\ &\leq e^{-u^{1-\delta}}
  \qquad \textrm{for some $\delta <1$, by Lemmas~\ref{p2}
    and~\ref{p3}},
\end{align*}
which is $o\left(\frac{1}{u^k}\right)$ for any $k>0$.

Since once again the estimate of the lower half is the same, this completes the
analysis for Zone~2.

\noindent \textbf{Zone 3:} Because $P\left(T_\beta=0\right)=e^{-\int_0^1
  e^{\beta x} g(x) dx}$,  the integral \eqref{eq206} in  the upper half of Zone 3 is 
\begin{equation}\label{30} 
P\left(T_\beta=0\right)\lim_{A \rightarrow \infty}
\frac{1}{2\pi}\int_{\beta \sigmabeta^2}^A e^{i(y+\alpha
  u)t}\left\{e^{\int_0^1 e^{(\beta-i \alpha t) x} g(x)
  dx}-1\right\}dt.
\end{equation}
By Lemma~\ref{p3},
\begin{equation}\label{eq2125}
P\left(T_\beta=0\right)=O\left(e^{-Ku}\right)
\end{equation}
for some $K>0$.  As for the rest of (\ref{30}), for fixed $\beta$ we
have
\[
\lim_{t \rightarrow \infty} \int_0^1 e^{(\beta-i\alpha t)x}g(x)dx=0,
\]
by the Riemann--Lebesgue lemma.  So, setting
\begin{equation}\label{eq213}
I=I\left(\beta,t\right)=\int_0^1 e^{(\beta-i\alpha t)x} g(x)dx,
\end{equation}
this encourages us to write
\begin{equation}\label{eq214}
e^{I(\beta,t)}-1=I(\beta,t)+O\left(I^2\left(\beta,t\right)\right)
\end{equation}
and substitute \eqref{eq214} for the braced expression in
(\ref{30}).   
But we need more detailed information about $I(\beta, t)$ before we
can exploit the oscillatory factor $e^{i(y+\alpha u)t}$ in (\ref{30}). 

\begin{lem}\label{L3}
There are finitely many numbers
\[
0\leq a_0<a_1<\ldots<a_L=1
\]
and
\[
C_0,\ldots,C_L,
\]
with $C_L \ne 0$, such that
\begin{equation}\label{eq215}
I(\beta,t)=\sum_{j=0}^L C_j \frac{e^{a_j(\beta-i\alpha
    t)}}{\beta-i\alpha t}+O\left(\frac{\sigmabeta^2 e^{2
    \beta}}{t^2}\right).
\end{equation}
\end{lem}

\begin{proof} 
This is just integration by parts.  Let
\[
0\leq a_0<a_1\cdots<a_L=1
\]
be the discontinuity points of $g$.  The contribution to $I$ from the
subinterval $\left[a_j,a_{j+1}\right]$ is
\[
\int_{a_j}^{a_{j+1}} e^{(\beta-i\alpha t)x}g(x)dx
\]
\[
=\frac{1}{\beta-i\alpha t}\left(g(a_{j+1})e^{(\beta-i\alpha
  t)a_{j+1}}-g(a_j)e^{(\beta- i \alpha t)a_j}\right) -\frac{1}{\beta-i
  \alpha t} \int_{a_j}^{a_{j+1}} e^{(\beta-i\alpha t)x} g^{\prime}
(x)dx
\]
where $g\left(a_j\right)$ and $g\left(a_{j+1}\right)$ are evaluated as
one-sided limits, where necessary, and
\[
\frac{1}{(\beta-i\alpha t)}\int_{a_j}^{a_{j+1}}e^{(\beta-i \alpha
  t)x}g^\prime (x) dx =O\left(\frac{\sigmabeta^2 e^\beta}{t^2}\right)
\]
via another integration by parts.  (This is where we use the
hypothesis that $g$ is piecewise $C^2$!)  Now collect all the terms
from all the subintervals.   This completes the proof of 
Proposition  \ref{new main prop}.
\end{proof}

Having Lemma \ref{L3}, 
we can use \eqref{eq215} as follows.  For $a \leq
1$ and for any fixed $y$ we have

\begin{align}\label{eq216}
&\lim_{A \rightarrow \infty} \int_{\beta \sigma_{\beta}^2 }^A
  e^{i(y+\alpha u)t} \frac{e^{a(\beta-i \alpha t)}}{\beta-i \alpha t}
  dt\\ &=\lim_{A\rightarrow \infty}\sigma_{\beta} e^{a \beta}
  \int_{\beta \sigma_{\beta}}^A \frac{e^{ict}}{\beta-it}dt\qquad
  (\textrm{where $c=y/\alpha+u-a$}) \\ &=\lim_{A \rightarrow \infty}
  \sigma_{\beta} e^{a \beta} \left\{\frac{e^{ict}}{ic(\beta-it)}
  \mathop{\bigg|}\nolimits_{\beta \sigma_{\beta}}^A +\int_{\beta
    \sigmabeta}^A
  \frac{e^{ict}}{c(\beta-it)^2}dt\right\}\\ &=\sigma_{\beta} e^{a
    \beta} O\left(\frac{1}{c(\beta-i \beta
    \sigma_{\beta})}\right)\\ &=O\left(\frac{\sigma_{\beta}
    e^\beta}{\beta u^{3/2}}\right)=O(1)
\end{align}
using Lemmas~\ref{p2} and~\ref{p3}.

Also, looking at the big-$O$ term in \eqref{eq215} gives
\begin{equation}\label{eq217}
\lim_{A \rightarrow \infty} \int_{\beta \sigmabeta^2}^A e^{i(y+\alpha
  u)t} O\left(\frac{\sigmabeta^2 e^{2
    \beta}}{t^2}\right)dt=O\left(\frac{e^{2 \beta}}{\beta}\right).
\end{equation}
Further, a little algebra with \eqref{eq215} shows that
\[
I^2(\beta,t)=O\left(\frac{\sigmabeta^4 e^{4 \beta}}{t^2}\right)
\]
is a very generous bound.  This gives
\begin{equation}\label{eq218}
\lim_{A \rightarrow \infty} \int_{\beta \sigma_{\beta}^2}^A
e^{i(y+\alpha u)t} O\left(I^2(\beta,t)\right)
dt=O\left(\frac{\sigmabeta^2 e^{4 \beta}}{\beta}\right).
\end{equation}
Obviously, \eqref{eq218} is the dominant bound. The quantity in \eqref{eq216} is so
small because it ``sees'' the oscillation.

Finally, substituting \eqref{eq214} into \eqref{30} and using the
bounds just calculated yields a bound of $O\left(e^{-Ku}
\frac{\sigmabeta^2 e^{4 \beta}}{\beta}\right)$ for Zone 3, which is
certainly $o\left(\frac{1}{u^k}\right)$ for any $k$.

The lower half of the zone is handled in exactly the same way.  This
completes the proof of Proposition \ref{new main prop}.

\end{proof}

\section{Proof of Theorem \ref{new theorem 1}}\label{5}

\begin{proof}
Given Proposition \ref{new main prop} the proof of Theorem \ref{new theorem 1}
is relatively straightforward.  Let
$$ C_g(\beta)=\int_0^1 \left(e^{\beta x}-1\right)g(x)dx
$$ and, given $u>0$, let $\beta=\beta(u)$ be the solution of
$$
  C_g'(\beta)=\int_0^1 x \, e^{\beta x} \, g(x) \ dx = u.
$$
By  Lemmas~\ref{p1},~\ref{p2}(b)   and~\ref{p3}(b) when $\beta=\beta(u)$
we then have 
$u=E_\beta := E T_\beta$, while
the variance $\sigmabeta^2=$ {\rm var} $ T_\beta$ satisfies both
$\sigma^2_\beta = 1/\beta'(u)$, and 
$ \sigmabeta^2 \sim u$.
As usual we define $\sigmabeta$ to be the positive square root.

Since $u=E_\beta \rightarrow \infty$ as $\beta \rightarrow \infty$, for the inverse function we also have $\beta(u)\rightarrow \infty$ as $u \rightarrow \infty$.  Just as in the Dickman case, discussed in Section~\ref{1}, we have 
\begin{equation}\label{eq2035}
f_Y(y)=\sigmabeta f_\beta \left(\sigmabeta y+u\right),
\end{equation} 
and, with $Ee^{\beta T_0} =\exp(\int_0^1 \left(e^{\beta
  x}-1\right)g(x)dx)=\exp(C_g(\beta))$
we have
\begin{equation}\label{eq2036}
f_\beta(t)=e^{\beta t}f(t)/\exp(C_g(\beta)).
\end{equation} 
We combine the two equations above, to express the density $f$ of $T_0$ at
the point $t=u+\sigmabeta y$ 
which is, for $T_\beta$, $y$ standard deviations above its
mean, $u$:
\begin{eqnarray}\label{untilt}
 f(u + \sigmabeta \, y)& = &e^{C_g(\beta)-u\beta } e^{-\beta
  \sigmabeta y} \ f_\beta(u + \sigmabeta \, y) \\
  &=&
\label{nice form density}
 e^{C_g(\beta)-u\beta } e^{-\beta
  \sigmabeta y}\frac{1}{\sigmabeta} f_Y(y).
\end{eqnarray}

We now complete the proof of~\eqref{eq2032} by taking $y=0$
in~\eqref{nice form density} and then  using~\eqref{eq203}.

To prove \eqref{tail asymp}, we combine~\eqref{eq202}  
error term, i.e., the standard  
$f_Y(y)$, with~\eqref{untilt} and \eqref{nice form density}, to find
\begin{equation}
\p( T_0 \ge u ) = \sigmabeta \int_0^{\infty} f(u+\sigmabeta y)\ dy
\end{equation}
\begin{equation}\label{step2}
=
 \frac{e^{C_g(\beta)-u\beta}}{\sqrt{2 \pi}} \int_0^{\infty} e^{-\beta
   \sigmabeta y} \left[e^{-y^2/2} + O(1/\sqrt{u}) \right] dy.
\end{equation}

The main term of the integral in~\eqref{step2} is $\int_0^{\infty}
e^{-\beta \sigmabeta y}\, e^{-y^2/2} \, dy = 1/(\beta \sigmabeta) (
1+O(1/u))$; we see this approximation in two steps. The upper bound is
simply $\int_0^{\infty} e^{-\beta \sigmabeta y}\, e^{-y^2/2} \, dy <
\int_0^\infty e^{-\beta \sigmabeta y} \ dy = 1/(\beta \sigmabeta)$.  A
lower bound, for any $0<d$, is $\int_0^{\infty} e^{-\beta \sigmabeta
  y}\, e^{-y^2/2} \, dy \ge e^{-d^2/2} \int_0^d e^{-\beta \sigmabeta
  y}\ dy = 1/(\beta \sigmabeta) e^{-d^2/2}(1-e^{-d \beta \sigmabeta})
$.  We can use $d=1/\sigmabeta$, so that
$e^{-d^2/2}=1-O(1/\sigmabeta^2)=1-O(1/u)$ and $(1-e^{-d \beta
  \sigmabeta}) =1-e^{-\beta}=1-o(1/u)$, using Lemma~\ref{p2}(b).

To bound the error term in  the integral in~\eqref{step2}, use
the uniformity of the error term in~\eqref{eq202}.  After
applying absolute value and taking the absolute value inside, we have
an upper bound of the form, with some  fixed $K< \infty$, 
$\int_0^{\infty} e^{-\beta
 \sigmabeta y} K/\sqrt{u} \ dy = 
 1/(\beta \sigmabeta) \times 
O(1/\sqrt{u}).$  This completes the proof of~\eqref{tail asymp}.   
\end{proof}

\section{Proof of Theorem \ref{new theorem 2}}\label{sect end run}

We will actually derive Theorem \ref{new theorem 2} as a corollary of 
Theorem \ref{new theorem 1}  
 by means of the following strategy:  given a density $g$, satisfying the hypotheses of 
Theorem \ref{new theorem 2},
which blows up at $0+$, consider, for some fixed $0<a<1$, the Poisson
process restricted to $[a,1]$.  Write $T^\mrk2$ for the sum of
  arrivals
in $[a,1]$;  this is a process to which both
Lemma \ref{p4} and Proposition \ref{new main prop}  apply.

The sum $T^\mark1$ of arrivals in $(0,a)$ is independent of $T^\mrk2$;
we have $T=T^\mark1 + T^\mrk2$, and $T^\mark1$ is
relatively small.  The distribution of $T$ is given by convolving
the distribution of $T^\mark1$, over which we have relatively little
control,
with the distribution of $T^\mrk2$,  which has an atom at 0, and a
density
well-controlled by the above-mentioned results.  
What, then, can we say about the density of $T$?  Since $T^\mrk2$ has an
atom at 0, and the density of $T^\mark1$ may be unbounded at $0+$, a uniform
\emph{everywhere} approximation for the density of $T$, such as \eqref{eq202}
is not possible---it fails at $y=-u/\sigmabeta$.  But to prove
our result, 
we will only need the uniform
approximation at $y \ge 0$.

It was clearly necessary to require that 
$\int_0^1 x g(x) \ dx < \infty$, 
so that $T$ is a random variable with finite values; 
and we have, in fact, imposed 
a slightly more restrictive hypothesis, namely that 
$ \sup_{0<x \le 1} x g(x) < \infty$.  This contrasts with 
Lemma~\ref{p4}, in which we allow $\int_0^1 g(x) \ dx$ to be either
finite or
infinite.

\begin{lem}\label{lemma size bias}
Suppose that $L := \sup_{0<x \le 1} x g(x) < \infty$.  Then $T^\mark1$ has a
(possibly defective) density $f$, with $\sup_{0<x \le 1} x f(x) \le L
< \infty$.
\end{lem}
\begin{proof}
Since $T^\mark1$ is a sum of Poisson arrivals, its
size-biased distribution  
is that of a random variable $T^*$
which
satisfies
\[
  T^* = T^\mark1+I \mbox{ in distribution},
\]
with $T^\mark1,I$ independent  and $I$ having density $x g(x) /c$, where
$c=\int_0^1 x g(x) \ dx$---
see~\cite{AGK}. 
Since the distribution of $T^*$ is the convolution of a probability 
distribution, namely that of $T^\mark1$, with an absolutely continuous
distribution, namely that of $I$, the distribution of $T^*$ has a proper
probability density, say $f^*$.  We infer from the convolution equation $ T^* = T^\mark1+I$
that even if the distribution of $T^\mark1$ has an
atom at 0, it must also have a (possibly defective) density $f$ on
$(0,\infty)$
satisfying $x f(x)/c = f^*(x)$.  Then the convolution equation at the density level reduces,
after multiplication by a factor of $c$ on both sides, to
\[
     x f(x) = \int_0^1 f(x-z) z g(z) \ dz
\]
which implies that $\sup x f(x) \le \sup x g(x)$.
\end{proof}

We next give an upper bound for use with the ``relatively small''
contributions
from $T^\mark1_\beta$. 

\begin{lem}\label{easy ueub}
Given $ 0 < m < \infty$, let $\nu$ be any measure on $(0,1]$ such that 
$m= \int x \ d \nu$, and let $W$ be the sum of arrivals in the Poisson process 
$\pp(\nu)$, so that $EW=m$.  Uniformly over choices of $\nu$, the
  chance that $W$ exceeds twice its mean  decays exponentially fast
  relative to $m \to \infty$; in fact
\[
   P( W \ge 4 \, EW ) \le \exp(-m(4-e)).
\]
\end{lem}
\begin{proof}
Since $\nu$ is supported by (0,1], $E e^W = \exp( \int (e^x-1) d \nu)
\le$ $ \exp( \int e \, x \ d\nu) = \exp(em)$.  Hence
\[ 
 P( W \ge 4 \, EW )= P(e^W \ge e^{4m}) \le E e^W / e^{4m} \le
 e^{em}/e^{4m}.
\]
\end{proof}

We apply Lemma~\ref{easy ueub} to the random variables $T^\mark1_\beta$
to show that the chance of strictly exceeding 4 times the mean is $o(1/u)$, where $u$ is defined
by $\beta = \beta(u)$ ---
if $T^\mark1=0$ identically,  the probability is zero. Otherwise
there is some $\epsilon>0$ with $\int_{2 \epsilon}^a g(x) \ dx)>\epsilon$.
 When we tilt $T$ to get $E T_\beta=u$,  by
Lemma~\ref{p3}(b), we know that $\beta \sim \log u$, hence for
sufficiently large $u$ we have $\beta > .5 \log u$, and hence $E T^\mark1_\beta \ge
\epsilon^2 u^\epsilon$. The upper bound from  Lemma~\ref{easy ueub} 
then gives 
\begin{equation}\label{dag upper tail}
P( T^\mark1_\beta > 4 E T^\mark1_\beta) \le \exp(-(4-e)\epsilon^2
u^\epsilon)=
o(1/u). 
\end{equation}

 We can now prove Theorem \ref{new theorem 2}:
\begin{proof}[Proof of Theorem \ref{new theorem 2}]
As outlined in the second paragraph of this section, write 
 $T=T^\mark1 + T^\mrk2$  for the sum of arrivals in the Poisson process
 with
arrival intensity $g(x) dx$.  Write $T_\beta=T^\mark1_\beta + T^\mrk2_\beta$
for the same, after tilting.

We apply Proposition \ref{new main prop}  
to $T^\mrk2$, using the
 $\beta$  
for which $u = E T_\beta$. We must be careful
to note that the mean and variance used to standardize  $T^\mrk2_\beta$
are \emph{not} the mean $u$ and variance $\sigmabeta^2$ of $T_\beta$.
To emphasize this, we write
\[
  u^\mrk2 :=  E T^\mrk2_\beta = \int_a^1 x e^{\beta x} g(x) \ dx, \ \ \ 
\sigmabeta^{\mrk22} =   \int_a^1 x^2 e^{\beta x} g(x) \ dx,
\]
so that \eqref{eq201}, specifying the standardized version of 
$T^\mrk2_\beta$, is
\begin{equation}\label{Y ddag}
  Y =\frac{T^\mrk2_\beta -  u^\mrk2}{\sigmabeta^\mrk2}.
\end{equation}

If we write  
 
\[
  u^\mark1 :=  E T^\mark1_\beta = \int_0^a x e^{\beta x} g(x) \ dx, \ \ \ 
\sigmabeta^{\mark12}  = \int_0^a x^2 e^{\beta x} g(x) \ dx,
\]
so that $u=u^\mark1 + u^\mrk2$ and 
$\sigmabeta^2=\sigmabeta^{\mark12}+\sigmabeta^{\mrk22}$   then we see that
\begin{equation}\label{u ddag estimate}
u^\mrk2 = u - u^\mark1 = u - O(u^a) = u \ ( 1 + O(u^{a-1}))
\end{equation}
and
\[
\sigmabeta^{\mrk22}=\sigmabeta^2-\sigmabeta^{\mark12}= \sigmabeta^2-
O(u^a)
\sim u \ ( 1 + O(u^{a-1})),
\]
so
\begin{equation}\label{sigma ddag estimate}
\sigmabeta^{\mrk2}\sim \sqrt{u} \ ( 1 + O(u^{a-1})).
\end{equation}
We write $f^\mrk2_\beta$ for the (defective) density of
$T^\mrk2_\beta$, and $f_Y$ for the (defective) density for the $Y$ in
\eqref{Y ddag}.  The ordinary change of variables, together with 
Proposition \ref{new main prop}
give, uniformly in $y$,
\begin{equation}\label{f ddag}
f^\mrk2_\beta(x) = \frac{1}{\sigmabeta^{\mrk2}} 
f_Y\left(\frac{x-u^\mrk2}{\sigmabeta^{\mrk2}}\right) \mbox{ and }
f_Y(y)=\frac{1}{\sqrt{2 \pi}}e^{-y^2/2} 
+O\left( \frac{1}{\sqrt{u^\mrk2}}\right).
\end{equation}

Since  $T_\beta=T^\mark1_\beta + T^\mrk2_\beta$ with independent non-negative
summands, we have (see for example~\cite[page 356]{ash})  
\begin{equation}\label{f beta conv}
f_\beta(t) = P(T^\mrk2_\beta=0) f^\mark1_\beta(t) + 
\int f^\mrk2_\beta(t-z) dF^\mark1_\beta(z).
\end{equation}
[We took the integral $ dF^\mark1_\beta(z)$ rather than with
  $f^\mark1_\beta(z) \ dz$ since $T_\beta^\mark1$ may have an atom at 0.]
By Lemma~\ref{lemma size bias}, $\sup t f(t) \le L$, so for $t>1$,
$f(t) \le L$ and $f_\beta(t) \le L e^{\beta t}$.  The size of the atom
at 0 for $T^\mrk2_\beta$ is $\exp(-\int_a^1 e^{\beta x} g(x) \ dx)
\approx \exp(-u)$, so uniformly in $1 \le t \le u$ we have $
P(T^\mrk2_\beta=0) f^\mark1_\beta(t)=o(1/u)$, and this contribution
to \eqref{f beta conv} is 
negligible.

To complete the proof of~\eqref{eq2032 kludge} and \eqref{tail asymp kludge},
arguing as in the two paragraphs 
 following~\eqref{step2}, we will only need
approximations for $f_\beta(t)$ with $t \ge u$, and especially with
  $t \in [u,u+1]$.  Also, we will
use
Lemma~\ref{easy ueub} to bound the contribution from $
dF^\mark1_\beta(z)$ with $z \ge 4 u^\mark1$.
Hence, we consider $t \in [u,u+1]$ and $z \in [0,4 u^\mark1]$ so that
$
x := t-z \in [u - 4 u^\mark1,u+1]$, hence 
$$ x-u^\mrk2 \in
 [u-u^\mrk2 - 4 u^\mark1,u-u^\mrk2+1] = [- 3 u^\mark1,u^\mark1+1].
$$
Combining these bounds with \eqref{u ddag estimate}
and \eqref{sigma ddag estimate}, uniformly over
$t \in [u,u+1]$ and $z \in [0,4 u^\mark1]$, the argument
$y :=(x-u^\mrk2)/\sigmabeta^\mrk2$ to $f_Y$ in 
\eqref{f ddag} is $O(u^{a-.5})$.  Hence $y^2=O(u^{2a-1})$, which is
$O(1/\sqrt{u})$ using $a \le .25$.  Combining with the error term
written as $O(1/\sqrt{u^\mrk2})$, and  using $u^\mrk2 \sim u$, we
have,  uniformly over
$t \in [u,u+1]$
and $z \in [0,4 u^\mark1]$,
\begin{equation}\label{f ddag approx}
  f_\beta^\mrk2(t-z) = \frac{1}{\sqrt{2 \pi} \, \sigmabeta^\mrk2}
  (1+O(1/\sqrt{u})) 
=  
 \frac{1}{\sqrt{2 \pi} \, \sigmabeta}
  (1+O(1/\sqrt{u})).
\end{equation}
The second equality in \eqref{f ddag approx}, 
switching from $\sigmabeta^\mrk2$ to 
$\sigmabeta$ in the denominator, is justified by~\eqref{sigma ddag
  estimate}, 
together with $a \le .5$, to combine two
error terms into one.

The contribution to \eqref{f beta conv} from integrating over 
$z > 4 u^\mark1$ is $o(1/u)$, using~\eqref{dag upper tail} 
and the fact that $f^\mrk2$ is bounded by $\sup_{a\le x \le 1} g(x)$.
Finally, the contribution to~\eqref{f beta conv}
from integrating over 
$z \in [0, 4 u^\mark1]$ is $\frac{1}{\sqrt{2 \pi} \, \sigmabeta^\mrk2}
  (1+O(1/\sqrt{u}))$, using \eqref{f ddag approx}, and  using
 \eqref{dag upper tail} in the form $\p(T^\mark1_\beta \in [0,4
   u^\mark1])=1-O(1/u)$.  Thus we have proved that uniformly over $t \in [u,u+1]$
\begin{equation}\label{almost done}
   f_\beta(t) =  \frac{1}{\sqrt{2 \pi} \, \sigmabeta}
  (1+O(1/\sqrt{u})).
\end{equation}

A similar argument, now for $t \ge u$ instead of for $t \in [u,u+1]$,
shows that uniformly over $t \ge u$, $
 \sigmabeta f_\beta(t) \le   1+O(1/\sqrt{u})$. 

We apply \eqref{untilt}, 
and argue as we did following~\eqref{nice form density}, 
to complete the proof of~\eqref{eq2032 kludge} and \eqref{tail asymp kludge}.
\end{proof}

\section{A Concluding Conjecture}\label{sect conclusions}

In our results we have assumed that the measure $\mu$ had a density
$g(x)$ satisfying certain regularity conditions.  In fact we
conjecture that this result can be considerably weakened.

\begin{conjecture}\label{conjecture 1}
Let $\mu$ be any measure on $(0,1]$ such that $\int x \ d \mu <
  \infty$.  Allow $\mu$ to have atoms and a continuous singular part,
  but require that the \emph{absolutely continuous} part of $\mu$, say
  $g(x) dx$, be nontrivial, in the sense that for some $0 \le a < b
  \le 1$ and $\varepsilon >0$ we have $g(x) \ge \varepsilon$ for all
$x \in (a,b)$.  Then the tilted sum of arrivals, $T_\beta$, 
as given by \eqref{cramer}
has standardized version $Y$ with density satisfying
\begin{equation}\label{eq202 weaker}
f_Y(y)\to \frac{1}{\sqrt{2 \pi}}e^{-y^2/2}
\end{equation}
for all $y$, 
similar to~\eqref{eq203} in Proposition \ref{new main prop},  
but we make no
assertion about a rate of convergence.
Hence   $T$  as given by \eqref{C from mu},
has (possibly defective) density $f$ satisfying
\begin{equation}\label{eq2032 weaker}
f(u)= \sqrt{\frac{\beta^\prime (u)}{2
    \pi}}e^{C(\beta)-u\beta }\left(1+o\left(1 \right) \right)
\end{equation}
where
\[
C(\beta)=\int_0^1 \left(e^{\beta x}-1\right) \mu(dx),
\]
similar to \eqref{eq2032} in Theorem \ref{new theorem 1},  
 again with no assertion concerning a rate of convergence.
\end{conjecture}

\noindent \textit{Acknowledgement}: We are grateful to Charles Fefferman for
helpful comments concerning Lemma~\ref{p4}.




\end{document}